\documentclass[10p]{amsart}
\usepackage[margin=1in]{geometry}
\usepackage{version,amsmath, amssymb, mathtools,color}
\usepackage{color}
\usepackage{amsthm, amsopn, amsfonts, xypic}
\usepackage[colorlinks=true]{hyperref}
\usepackage{mathrsfs}
\usepackage{slashed}
\hypersetup{urlcolor=blue, citecolor=blue}

\usepackage[color=yellow]{todonotes}

\let\arXiv\arxiv
\def\doi#1{ {\href{http://dx.doi.org/#1}
   {{\mdseries\ttfamily DOI}}}}

\def\lf{\left\lfloor}   
\def\rf{\right\rfloor}

\newcommand{\beq}{\begin{equation}}\newcommand{\eq}{\end{equation}}

\renewcommand{\tt}{ \tilde t}

\newcommand{\pa}{{\partial}}

\newcommand{\at}{{\tilde t}}

\newcommand{\ga}{\gamma}
\newcommand{\tG}{{\widetilde{G}}}
\newcommand{\tpsi}{{\widetilde{\psi}}}

\newcommand{\newsection}[1]
{\section{#1}\setcounter{theorem}{0} \setcounter{equation}{0}
\par\noindent}

\newtheorem{theorem}{Theorem}

\newtheorem{lemma}[theorem]{Lemma}

\newtheorem{proposition}[theorem]{Proposition}

\newcommand{\N}{{\mathbb N}}

\newcommand{\R}{{\mathbb R}}

\newcommand{\ang}{{\not\negmedspace\partial }}

\newcommand{\la}{\langle}
\newcommand{\ra}{\rangle}

\newcommand{\rs}{{\tilde r}}

\renewcommand{\S}{{\mathbb S}}
\newcommand{\M}{\mathcal M}

\newcommand{\weight}{{\Bigl(1-\frac{2M}{r}\Bigr)}}

\begin{document}

\title
{
Pointwise decay for semilinear wave equations on Kerr spacetimes
}

\author{Mihai Tohaneanu}
\address{Department of Mathematics, University of Kentucky, Lexington,
  KY  40506}
\email{mihai.tohaneanu@uky.edu}

\begin{abstract}
   In this article we prove pointwise bounds for solutions to the semilinear wave equation with integer powers $p\geq 3$ on Kerr backgrounds with small angular momentum and small initial data. We expect that the bounds proved in this paper are optimal.
 \end{abstract}

\maketitle

\tableofcontents

\newsection{Introduction}

\subsection{Statement of the result}

 In this paper we establish pointwise bounds for solutions to the semilinear wave equations on Kerr backgrounds. We consider the equation
\begin{equation}\label{nlwintro}
\Box_K \phi = \pm \phi^p, \qquad \phi |_{\at=0} = \phi_0, \qquad \tilde T \phi |_{\at=0} = \phi_1\ .
\end{equation}

Here the coordinate $\at$ is chosen so that the slice $\at=0$ is space-like and so that $\at=t$
away from the black hole, and $\tilde T$ is the future unit normal to $\at=0$ (see Section 2 for details). Moreover, $p\geq 3$ is any integer, $\Box_K$ denotes the d'Alembertian in the Kerr metric, and the initial data are smooth and supported in $\{|x| \leq R_1\}$ for some fixed (but arbitrary) $R_1$. We also assume that, for some fixed $N\gg 1$, we have that
\[
\|\phi_0\|_{H^{N+1}} + \|\phi_1\|_{H^N} \leq \varepsilon.
\]

Let $\kappa = \min\{2, p-2\}$. Let $\rs$ equal to $r$ on a compact region and to the Regge-Wheeler coordinate $r^*$ near infinity (see Section 2 for details). We will show that, for small enough $\varepsilon$,
\beq\label{main}
\phi \lesssim \frac{\varepsilon}{\la \at\ra \la \at-\rs\ra^\kappa},
\eq
where $\la x\ra = \sqrt{2+|x|^2}$. 

In future work, we will also show that this rate of decay is sharp.

\subsection{History}

The semilinear wave equation in $\R^{1+3}$
\[
\Box \phi = \pm \phi^p, \qquad \phi |_{t=0} = \phi_0, \qquad \pa_t \phi |_{t=0} = \phi_1
\]
has been studied extensively. There are many papers concerned with global existence, uniqueness, and scattering, see for example \cite{Jor}, \cite{Str}, \cite{Pec2}, \cite{ShStr}, \cite{BSh}, \cite{Gri}. In particular, it is well-known that for small initial data there is a unique global solution if $p> 1+\sqrt{2}$, see \cite{John}, \cite{GLS}, \cite{Ta}.

In terms of pointwise decay of solutions, there are a number of results, see for example \cite{Pec1}, \cite{Str}, \cite{Yang}. In the case of compactly supported smooth data, the optimal decay rate is
\[
\phi \lesssim \frac{1}{\la t\ra \la t-r\ra^{p-2}} \ .
\]
This was shown for small initial data in \cite{Sz} and for large data in the defocusing case in \cite{Gri} (when $p=5$) and \cite{BSz} (when $3\leq p < 5$).

There is also a vast literature concerned with establishing pointwise decay estimates for solutions to the linear wave equation $\Box_g \phi = 0$ for various Lorentzian metrics $g$. In the case of the Schwarzschild metric, the solution to the wave equation was conjectured to decay at the rate of $\tt^{-3}$ on a compact region by Price \cite{Pri}, and this rate of decay was shown to hold for a variety of spacetimes, including Schwarzschild and Kerr with $|a|<M$, see \cite{DSS}, \cite{Tat}, \cite{MTT}, \cite{LOh}, \cite{Hin}, \cite{AAG3}. It is by now well understood that once local energy estimates  is established in a compact region on an asymptotically flat region, one can obtain pointwise decay rates that are related to how fast the metric coefficients decay to the Minkowski metric; see, for example, \cite{Tat}, \cite{MTT}, \cite{OS}, \cite{Mos}, \cite{AAG1}, \cite{AAG2}, \cite{Mor}, \cite{MW}, \cite{Looi1}.

Much less is known of the behavior of solutions of the semilinear wave equations with power nonlinearity on Schwarzschild and Kerr backgrounds. For small initial data, global existence and pointwise decay rates of $t^{-1}$ were shown for spherically symmetric solutions in Schwarzschild in the case $p>4$ in \cite{DR1}; a similar result was shown the case $p>3$ (without the spherically symmetric assumption) in \cite{BSt}. For Kerr with small angular momentum and $p>1+\sqrt{2}$, global existence was shown in \cite{LMSTW}. 

The goal of this paper is to establish sharp decay rates for solutions to \eqref{nlwintro} with small initial data. For simplicity, we also pick the initial data to be compactly supported, though one can do away with it assuming enough decay in weighted Sobolev spaces. The ideas of the proof also apply to other metrics and nonlinearities; in particular, see the upcoming result of Looi \cite{Looi2}, where a similar result for the quintic defocusing nonlinearity on perturbations of Minkowski (and large initial data) is shown. 

The paper is structured as follows. In Section 2 we introduce the Kerr metric, our preferred coordinates, local energy estimates, and our main theorem. Section 3 is dedicated to proving local energy estimates for our nonlinear problem. These estimates imply an initial pointwise decay rate of $\frac{(t-r)^{1/2}}{t}$, which is insufficient as a starting point for $p\leq 5$. Section 4 supplements this with $r^p$ estimates, which give better weighted $L^2$ estimates. Section 5 rephrases the problem in a more convenient way. Section 6 contains the proof of the main lemma used to improve the pointwise decay. Section 7 uses the results of the previous sections to yield an initial pointwise decay of $\frac{1}{r(t-r)^{\delta}}$. Section 8 obtains improved decay rates for the solution in the interior, and for derivatives. Section 9 is the main bootstrap argument: starting from the decay estimate of Section 7, we use the results of Sections 6 and Section 8 to improve the decay rate to the optimal one.

\subsection{Acknowledgments}

The author would like to thank Hans Lindblad, Sung-Jin Oh and Shi-Zhuo Looi for many useful conversations regarding the paper, and the Korea Institute for Advanced Study for their hospitality during the spring of 2019. The author was partly supported by the Simons collaboration Grant 586051.
 
 \medskip
\newsection{The Kerr metric and local energy norms}

\subsection{The Kerr metric}

The Kerr geometry in Boyer-Lindquist coordinates is given by
\[
ds^2 = g^K_{tt}dt^2 + g_{t\phi}dtd\phi + g^K_{rr}dr^2 + g^K_{\phi\phi}d\phi^2,
 + g^K_{\theta\theta}d\theta^2
\]
 where $t \in \R$, $r > 0$, $(\phi,\theta)$ are the spherical coordinates
on $\S^2$ and
\[
 g^K_{tt}=-\frac{\Delta-a^2\sin^2\theta}{\rho^2}, \qquad
 g^K_{t\phi}=-2a\frac{2Mr\sin^2\theta}{\rho^2}, \qquad
 g^K_{rr}=\frac{\rho^2}{\Delta},
 \]
\[ g^K_{\phi\phi}=\frac{(r^2+a^2)^2-a^2\Delta
\sin^2\theta}{\rho^2}\sin^2\theta, \qquad g^K_{\theta\theta}={\rho^2},
\]
with
\[
\Delta=r^2-2Mr+a^2, \qquad \rho^2=r^2+a^2\cos^2\theta.
\]

 Here $M$ represents the mass of the black hole, and $aM$ its angular momentum.

 A straightforward computation gives us the inverse of the metric:
\[ g_K^{tt}=-\frac{(r^2+a^2)^2-a^2\Delta\sin^2\theta}{\rho^2\Delta},
\qquad g_K^{t\phi}=-a\frac{2Mr}{\rho^2\Delta}, \qquad
g_K^{rr}=\frac{\Delta}{\rho^2},
\]
\[ g_K^{\phi\phi}=\frac{\Delta-a^2\sin^2\theta}{\rho^2\Delta\sin^2\theta}
, \qquad g_K^{\theta\theta}=\frac{1}{\rho^2}.
\]

The case $a = 0$ corresponds to the Schwarzschild space-time.  We shall
subsequently assume that $a$ is small $0 < a \ll M$, so that the Kerr
metric is a small perturbation of the Schwarzschild metric. Note also that the coefficients depend only $r$ and $\theta$ but are independent of
$\phi$ and $t$. 

We denote the Kerr metric by $g_K$, and the Schwarzschild metric by $g_S$. Let
$\Box_{K} $ and $\Box_{S} $ be the d'Alembertian associated to the Kerr and Schwarzschild metric, respectively. Similarly $dV_K$ and $dV_S$ are the volume forms, and $d\Sigma_K$ and $d\Sigma_S$ are the restrictions of the volume form to a hypersurface.

In the above coordinates the Kerr metric has singularities at $r = 0$,
on the equator $\theta = \pi/2$, and at the roots of $\Delta$, namely
$r_{\pm}=M\pm\sqrt{M^2-a^2}$. To remove the singularities at $r = r_{\pm}$ we
introduce functions $r_K^*=r_K^*(r)$, $v_{+}=t+r_K^*$ and $\phi_{+}=\phi_{+}(\phi,r)$ so that (see
\cite{HE})
\[
 dr_K^*=(r^2+a^2)\Delta^{-1}dr,
\qquad
 dv_{+}=dt+dr_K^*,
\qquad
 d\phi_{+}=d\phi+a\Delta^{-1}dr.
\]

Note that when $a=0$ the $r_K^*$ coordinate becomes the Schwarzschild Regge-Wheeler coordinate
\[
r^*=r+2M\log(r-2M).
\]

The Kerr metric can be written in the new coordinates $(v_+, r, \phi_+, \theta)$
\[
\begin{split}
ds^2= &\
-(1-\frac{2Mr}{\rho^2})dv_{+}^2+2drdv_{+}-4a\rho^{-2}Mr\sin^2\theta
dv_{+}d\phi_{+} -2a\sin^2\theta dr d\phi_{+} +\rho^2 d\theta^2 \\
& \ +\rho^{-2}[(r^2+a^2)^2-\Delta a^2\sin^2\theta]\sin^2\theta \,
d\phi_{+}^2,
\end{split}
\]
which is smooth and nondegenerate across the event horizon up to but not including
$r = 0$. We introduce the function
\[
\tt = v_{+} - \mu(r),
\]
where $\mu$ is a smooth function of $r$. In the $(\tt,r,\phi_{+},
\theta)$ coordinates the metric has the form
\[
\begin{split}
ds^2= &\ (1-\frac{2Mr}{\rho^2}) d\tt^2
+2\left(1-(1-\frac{2Mr}{\rho^2})\mu'(r)\right) d\tt dr \\
 &\ -4a\rho^{-2}Mr\sin^2\theta d\tt d\phi_{+} + \Bigl(2 \mu'(r) -
 (1-\frac{2Mr}{\rho^2}) (\mu'(r))^2\Bigr)  dr^2 \\
 &\ -2a (1+2\rho^{-2}Mr\mu' (r))\sin^2\theta dr d\phi_{+} +\rho^2
 d\theta^2 \\
 &\ +\rho^{-2}[(r^2+a^2)^2-\Delta a^2\sin^2\theta]\sin^2\theta
d\phi_{+}^2.
\end{split}
\]

On the function $\mu$ we impose the following two conditions:

(i) $\mu (r) \geq  r^*$ for $r > 2M$, with equality for $r >
{5M}/2$.

(ii)  The surfaces $\tt = const$ are space-like, i.e.
\[
\mu'(r) > 0, \qquad 2 - (1-\frac{2Mr}{\rho^2}) \mu'(r) > 0.
\]
As long as $a$ is small, we can use the same
function $\mu$ as in the case of the Schwarzschild space-time in \cite{MMTT}.

Let $\tilde r$ denote a smooth strictly increasing function (of $r$) that equals $r$ for $r\leq R$ and $r^*$ for $r\geq 2R$ for some large $R$. We will use the coordinates $(\tt, x^{i})$, where $x^i =  \rs\omega$. We use Latin
indices $i,j=1,2,3$ for spatial summation and Greek indices $\alpha,\beta=0,1,2,3$ for space-time summation.  By $\la r\ra$ we denote a smooth radial function which agrees with $r$ for
large $r$ and satisfies $\la r \ra \geq 1$. Note that, since $r\approx \rs$, we can use $r^k$ and $\rs^k$ interchangeably when defining our spaces of functions.

We fix $r_e$ satisfying $r_-<r_e<r_+$. The choice of $r_e$ is unimportant,
and for convenience we may simply use $r_e=M$ for all Kerr metrics
with $a/M\ll 1$. Let $\M =  \{ \tt \geq 0, \ r \geq r_e \}$, $ \Sigma(T) =  \M \cap \{ \tt = T \}$, and $ \Sigma^- =  \Sigma(0)$. 

A major difficulty in understanding dispersion properties for solutions to the linear wave equation on $\M$ is the presence of trapped null geodesics (i.e. null geodesics which do not escape either to infinity or to the singularity). One family of such geodesics occurs at the event horizon; however, due to the red-shift effect, the energy decays exponentially (in the high-frequency approximation) along such rays. A second family of such geodesics occurs in the compact region $|r-3M| \lesssim |a|$. Here the energy does not decay exponentially, but due to the hyperbolic nature of the trapping, it disperses after a time proportional to the logarithm of the frequency, and thus space-time estimates can still be recovered, albeit with a loss of derivatives.

Our favorite sets of vector fields will be
\[
\partial = \{ \partial_{\tt}, \partial_i\}, \qquad \Omega = \{x^i \partial_j -
x^j \partial_i\}, \qquad S = \tt \partial_{\tt} + \tilde r \partial_{\tilde r},
\]
namely the generators of translations, rotations and scaling. We set
$Z = \{ \partial,\Omega,S\}$.

For a triplet $\alpha=(i, j, k)$ we define $|\alpha| = i + 3j + 9k$ and
\[
u_{\alpha} = \pa^i \Omega^j S^k.
\]

The numerology is borrowed from \cite{MTT}, and takes into account the loss of derivatives that occurs when applying weak local energy estimates to vector fields.

We define the classes $S^Z(r^k)$ of
functions in $\R^+ \times \R^3$ by
\[
 f \in S^Z(r^k) \Longleftrightarrow 
|Z^j f(t, x)| \leq c_{j} \la r\ra^{k}, \quad j \geq 0.
\]
By $S^Z_{rad}(r^k)$ we denote spherically symmetric functions in $S^Z(r^k)$, and by $S^Z_{der}(r^k)$ the space of functions so that
\[
 f \in S^Z(r^k), \quad \pa f \in S^Z(r^{k-1}).
\]

Given a vector $g$, we will also use the notation
 \[
 f \in S^Z(r^k) g
 \]
 to mean that 
 \[
 f = \sum h_i g_i, \quad h_i\in S^Z(r^k),
 \]
 and similarly for $S^Z_{rad}(r^k)$, $S^Z_{der}(r^k)$.
 
In particular, a quick computation yields that
\beq\label{KSdiff}
g_K^{\alpha\beta} - g_S^{\alpha\beta} \in S^Z_{der}(r^{-2}).
\eq
Note that due to \eqref{KSdiff} we also have
\beq\label{K-Sc}
\Box_{g_K} u- \Box_{g_S} u \lesssim r^{-2} (|\pa^2 u| + |\pa u|).
\eq

The main theorem of the paper is the following:
\begin{theorem}\label{mainthm}
Let $p\geq 3$ be an integer. Assume that $\phi$ solves the wave equation
\begin{equation}\label{nlw}
\Box_K \phi = \pm \phi^p, \qquad \phi |_{\Sigma^-} = \phi_0, \qquad \tilde T \phi |_{\Sigma^-}= \phi_1\ .
\end{equation}
Let $\kappa = \min\{2, p-2\}$. Fix $m\in\N$ and $R_1 > r_e$. Then there are $N \gg m$ and $\varepsilon> 0$ so that, for any initial data $(\phi_0, \phi_1)$ supported in $\{r_e\leq r\leq R_1\}$ with
\[
\|\phi_0\|_{H^{N+1}} + \|\phi_1\|_{H^N} \leq \varepsilon,
\]
then $\phi$ exists globally in $\M$ and satisfies the pointwise bounds
\beq\label{mainbd}
\phi_{\leq m} \lesssim \frac{\varepsilon}{\la \at\ra \la \at-\rs\ra^\kappa} \ .
\eq
\end{theorem}

\subsection{Local energy norms}

We consider a partition
of $ \R^{3}$ into the dyadic sets $A_R= \{R\leq \la r \ra \leq 2R\}$ for
$R \geq 1$. We will use the notation $A\lesssim B$ to mean that there is a constant $C$ independent of $u$ and $\varepsilon$ so that $|A|\leq CB$; the value of $C$ might change from line to line. Similarly, $A \ll B$ means that $|A|\leq cB$ for a small enough constant $c$.

We now introduce the local energy norm $LE$ 
\begin{equation}
\begin{split}
 \| u\|_{LE} &= \sup_R  \| \la r\ra^{-\frac12} u\|_{L^2 (\M\cap \R \times A_R)}  \\
 \| u\|_{LE[\tt_0, \tt_1]} &= \sup_R  \| \la r\ra^{-\frac12} u\|_{L^2 (\M \cap [\tt_0, \tt_1] \times A_R)},
\end{split} 
\label{ledef}\end{equation}
its $H^1$ counterpart
\begin{equation}
\begin{split}
  \| u\|_{LE^1} &= \| \nabla u\|_{LE} + \| \la r\ra^{-1} u\|_{LE} \\
 \| u\|_{LE^1[\tt_0, \tt_1]} &= \| \nabla u\|_{LE[\tt_0, \tt_1]} + \| \la r\ra^{-1} u\|_{LE[\tt_0, \tt_1]},
\end{split}
\end{equation}
as well as the dual norm
\begin{equation}
\begin{split}
 \| f\|_{LE^*} &= \sum_R  \| \la r\ra^{\frac12} f\|_{L^2 (\M\cap \R\times A_R)} \\
 \| f\|_{LE^*[\tt_0, \tt_1]} &= \sum_R  \| \la r\ra^{\frac12} f\|_{L^2 (\M \cap [\tt_0, \tt_1] \times A_R)}.
\end{split} 
\label{lesdef}\end{equation}

 We also define similar norms for higher Sobolev regularity
\[
\begin{split}
  \| u_{\leq m}\|_{LE^1} &= \sum_{|\alpha| \leq m} \| u_{\alpha}\|_{LE^1} \\
  \| u_{\leq m}\|_{LE^1[\tt_0, \tt_1]} &= \sum_{|\alpha| \leq m} \| u_{\alpha}\|_{LE^1[\tt_0, \tt_1]} \\
  \| u_{\leq m}\|_{LE[\tt_0, \tt_1]} &= \sum_{|\alpha| \leq m} \| u_{\alpha}\|_{LE[\tt_0, \tt_1]},
\end{split}
\]
respectively 
\[
\begin{split}
  \| f\|_{LE^{*,k}} &=  \sum_{|\alpha| \leq k}  \| \partial^\alpha f\|_{LE^{*}} \\
  \| f\|_{LE^{*,k}[\tt_0, \tt_1]} &=  \sum_{|\alpha| \leq k}  \| \partial^\alpha f\|_{LE^{*}[\tt_0, \tt_1]}.
\end{split}  
\]

Finally, we introduce a weaker version of the local energy
decay norm \footnote {In Kerr one can actually control a stronger norm, where the $r$-derivative does not degenerate at the trapped set. However, we do not need the stronger norm in this paper.}
\[
\begin{split}
  \| u\|_{LE^1_{w}} &= \| (1-\chi_{ps}) \nabla u\|_{LE} + \| \la r\ra^{-1} u\|_{LE} \\
  \| u\|_{LE^1_{w}[\tt_0, \tt_1]} &= \| (1-\chi_{ps}) \nabla u\|_{LE[\tt_0, \tt_1]} + \| \la r\ra^{-1} u\|_{LE[\tt_0, \tt_1]},
\end{split}
\]
as well as the dual norms
\[
\begin{split}
 \| f\|_{LE^*_w} &= \| \chi_{ps} \nabla f\|_{L^2 L^2}+ \| f\|_{LE^*} \\
\| f\|_{LE^*_w[\tt_0, \tt_1]} &= \| \chi_{ps} \nabla f\|_{L^2[\tt_0, \tt_1] L^2}+ \| f\|_{LE^*[\tt_0, \tt_1]}.
\end{split}
\]

Here $\chi_{ps}$  is a smooth, compactly
supported spatial cutoff function that equals $1$ in a neighborhood of the trapped set. We also define the higher order weak norms as above.

We define the (nondegenerate) energy
\[
E[u](\tt) = \int_{\Sigma(\tt)} |\pa u|^2 d\Sigma_K(\tt).
\]

We will need the following local energy estimates, which were proved in \cite{TT} (for derivatives) and \cite{MTT} (for vector fields):
\beq\label{wle}
\| u_{\leq m}\|_{LE_w^1[\tt_0, \tt_1]} + \|\pa u_{\leq m}\|_{L^\infty L^2} \lesssim  E[u_{\leq m}](\tt_0) + \|(\Box_K u)_{\leq m}\|_{L^1L^2+LE^*_w[\tt_0, \tt_1]}.
\eq
We will also use a version that replaces the loss near the trapped set with a derivative loss:
\beq\label{lederiv}
\| u_{\leq m}\|_{LE^1[\tt_0, \tt_1]} + \|\pa u_{\leq m}\|_{L^\infty L^2} \lesssim  E[u_{\leq m+1}](\tt_0) + \|(\Box_K u)_{\leq m+1}\|_{L^1L^2+LE^*[\tt_0, \tt_1]}.
\eq

\medskip

\newsection{Local energy decay for the nonlinear problem}

In order to prove the main theorem from this section, we will use two results from  \cite{LMSTW}. We need the following weighted Sobolev
estimates, see Lemma 3.1 from \cite{LMSTW}:

For $R\ge 10$, $2\le q<\infty$, and any $b\in \R$, we have
\beq\label{ap-S}
 \|r^b v\|_{L_r^{\frac{2q(q-1)}{q-2}}L_{\omega}^{\infty}(r\ge R+1)}
 \lesssim  \|r^{b-\frac{1}{q-1}} v_{\leq 6}\|_{L_r^{q}L_{\omega}^{2}(r\ge R)},
\eq
\begin{equation}
  \label{ap-Sinfty}
  \|r^b v\|_{L^\infty_x(|x|\ge R+1)} \lesssim 
  \|r^{b-\frac{2}{q}} v_{\leq 6}\|_{L^q_r L^2_\omega(r\ge R)}.
\end{equation}

Moreover, we will use Theorem 3.2 from \cite{LMSTW}:

\begin{theorem}\label{thm-wStri} (Theorem 3.2, \cite{LMSTW})

Let $p\in [2,\infty)$.  Suppose $w$ solves
\[
\Box_K w=G_1+G_2, \qquad w(0,\cdot)=\pa_t w(0,\cdot) = 0,
\]
Additionally, suppose that $w$ vanishes in $[0,\infty]\times \{|x|\leq c\}$ for some $c>0$ .

Then for any $\delta_1>0$ and $1/2-1/q<s<1/2$ we have
\beq\label{weiStri}
 \| r^{\frac{3}{2}-\frac{4}q-s} w\|_{L^q L^q L^2}\lesssim
\|  r^{-\frac{1}{2}-s}  G_1\|_{ L^1L^1 L^2} + \| r^{\frac{3}{2}-s+\delta_1}G_2\|_{L^2L^2} .
 \eq
\end{theorem}

The goal of this section is to prove the following:
\begin{theorem}\label{LE}
Let $m\geq 6$ be a positive integer. Suppose that the initial data $(\phi_0, \phi_1)$ satisfies
\beq\label{Th-as}
\|(\phi_0)_{\leq m+6}\|_{H^1}^2+\|(\phi_1)_{\leq m+6}\|_{L^2}^2 \leq \varepsilon^2,
\eq 
where $\varepsilon$ is small enough. Then the equation \eqref{nlw} has a global solution that satisfies
\begin{equation}\label{derivloss}
\sup_{\tt} E[\phi_{\leq m}](\tt) + \| \phi_{\leq m}\|_{LE^1} \leq c_m \varepsilon
\end{equation}
\begin{equation}\label{nonlin}
 \| \phi_{\leq m}\|_{L^{p}L^{2p}} \leq c_m \varepsilon.
\end{equation}
\end{theorem}

\begin{proof}

 Let $\chi\in C^\infty (\R)$ satisfy $0\le \chi(r)\le 1$, $\chi(r)\equiv
0$ for $r\le R_1$, and $\chi(r)\equiv 1$ for $r>R_1+1$. Pick any $1+\sqrt{2}<q<3$.  For $\ga = \frac{4}{q}-\frac{2}{q-1} = \frac{2(q-2)}{q(q-1)}$, we define
\beq\label{X-norm} 
\| \phi\|_{X^m} = \| r^{-\ga} \chi \phi_{\leq m}  \|_{L^q L^q L^2} + \|\pa \phi_{\leq m}\|_{L^\infty L^2} + \|\phi_{\leq m-1}\|_{LE^1} 
\eq
\beq\label{N-norm}
 \|g\|_{N^m} = \| r^{- \ga q} \chi^q g_{\leq m} \|_{L^1 L^1 L^2} + \|g_{\leq m} \|_{L^1 L^2}.
\eq

We first prove the following linear estimate:
\begin{equation}\label{linest}
\|\phi\|_{X^m} \lesssim \|(\phi_0)_{\leq m+6}\|_{H^1}+\|(\phi_1)_{\leq m+6}\|_{L^2}+ \|\Box_K \phi\|_{N^m} \ .
\end{equation}

Indeed, the last two terms in \eqref{X-norm} can be estimated by using \eqref{lederiv}. 

In order to estimate the first term, we apply \eqref{weiStri} (with $s=\frac32-\frac{2}{q-1}$ and small $\delta_1$). 

Pick any $|\alpha|\leq m$. Note first that, due to the support properties of $\chi$ and the initial data, we have
\[
\chi \phi_{\alpha}(0, \cdot)= \pa_t \left(\chi \phi_{\alpha}(0, \cdot)\right) = 0.
\] 
 
 We have
\[
\Box_K (\chi \phi_{\alpha}) = \chi (\Box_K \phi)_{\alpha} + [\Box_K, \chi] \phi_{\alpha} + \chi [\Box_K, Z^{\alpha}] \phi
\]

An easy computation gives
\[
[\Box_K, \pa] \phi\in S^Z(r^{-2})\pa\pa^{\leq 1} \phi, \quad [\Box_K, \Omega]\phi \in S^Z(r^{-2})\pa\pa^{\leq 1} \phi,
\]
\[
[\Box_K, S] \phi \in S^Z(1) \Box_K \phi+ S^Z(r^{-2+})\pa \phi + S^Z(r^{-2+})\pa \Omega \phi + S^Z(r^{-2})\pa\pa^{\leq 1} \phi,
\]
and thus by induction we obtain that
\beq\label{comm}
[\Box_K, Z^{\alpha}] \phi = F_1+F_2, \quad F_1\in S^Z(1) (\Box_K \phi)_{\leq |\alpha|}, \quad F_2 \in S^Z(r^{-2+})\pa \phi_{\leq |\alpha|}.
\eq

We now pick $G_1 = \chi (\Box_K \phi)_{\alpha} + \chi F_1$ and $G_2=[\Box_K, \chi] \phi_{\alpha} + \chi F_2$. 

Since $\chi-\chi^q$ is supported in $[R, R+1]$, and $\frac{1}{2}+s = \ga q$, we see that
\[
\|  r^{-\frac{1}{2}-s}  G_1\|_{ L^1L^1 L^2} \lesssim \| r^{- \ga q} \chi (\Box_K \phi)_{\leq |\alpha|} \|_{L^1 L^1 L^2} \lesssim \| r^{- \ga q} \chi^q (\Box_K \phi)_{\leq |\alpha|} \|_{L^1 L^1 L^2} + \|(\Box_K \phi)_{\leq |\alpha|}\|_{L^1L^2}.
\] 
 
 On the other hand, we see that for small enough $\delta_1$:
\[
\|r^{\frac{3}{2}-s+\delta_1} \chi F_2 \|_{L^2L^2} \lesssim  \|\chi r ^{-\frac{1}{2}-s+\delta_1+} \pa\phi_{\leq |\alpha|}\|_{L^2L^2} \lesssim  \|\phi_{\leq|\alpha|}\|_{LE^1}. 
\]

We also have that
\[
\|r^{\frac{3}{2}-s+\delta_1} [\Box_K, \chi] \phi_{\alpha}\|_{L^2L^2} \lesssim \|\pa^{\leq 1}\phi_{\alpha}\|_{L^2L^2(R\leq |x|\leq R+1)} \lesssim  \|\phi_{\leq |\alpha|}\|_{LE^1}. 
\]

Theorem~\ref{thm-wStri} now implies \eqref{linest}.

We now finish the proof of Theorem~\ref{LE}. We want to show that
\beq\label{Xmbd}
\|\phi\|_{X^m} \leq C_m \varepsilon.
\eq

It is enough to show that, for any $\psi$,
\beq\label{nlnbd}
\| \psi^p\|_{N^m} \lesssim \|\psi\|_{X^m}^p.
\eq
Indeed, let $\psi_0\equiv 0$ and recursively define $\psi_{k+1}$
to be the solution to the linear equation
\beq\label{Picard}
  \Box_K \psi_{k+1} = \pm \psi_k^p , \qquad   \psi_{k+1}|_{\Sigma^-} = \phi_0, \qquad
  \tilde T \psi_{k+1}|_{\Sigma^-} = \phi_1.
\eq

 \eqref{linest} and \eqref{nlnbd} imply
\[
\|\psi_{k+1}\|_{X^m} \lesssim C(\epsilon + \|\psi_k\|_{X^m}^p),
\]
and a continuity argument implies that for small enough $\epsilon$ the sequence $\psi_k$ is Cauchy in $X^m$ and thus converges to a solution $\phi$ that satisfies \eqref{Xmbd}. 

Let us now prove \eqref{nlnbd}. We have for any $|\alpha|\leq m$:
\beq\label{nonlinvf}
Z^\alpha (\psi^p)\lesssim  \sum_{\substack{|\alpha_1|\leq \dots\leq|\alpha_p|\leq |\alpha| \\ \alpha_1 +\dots\alpha_p=\alpha}} |Z^{\alpha_1} \psi| \dots |Z^{\alpha_p} \psi|.
\eq

In a compact region, we have by H\"older and Sobolev embeddings
\[\begin{split}
& \Bigl\| Z^{\alpha_1} \psi \dots Z^{\alpha_p} \psi\Bigr\|_{L^1L^2(r\leq R+2)}  \leq \|Z^{\alpha_{p-1}} \psi\|_{L^2 L^\infty(r\leq R+2)} \|Z^{\alpha_p} \psi\|_{L^2 L^2(r\leq R+2)} \prod_{i=1}^{p-2} \|Z^{\alpha_i} \psi\|_{L^\infty L^\infty(r\leq R+2)} 
\\ & \lesssim  \|\psi_{\leq |\alpha_{p-1}|+2}\|_{L^2 L^2(r\leq R+3)}\|\psi_{\leq |\alpha_p|}\|_{L^2 L^2(r\leq R+2)} \|\pa \psi_{\leq |\alpha_{p-2}|+1}\|^{p-2}_{L^\infty L^2}\\ &  
 \lesssim \|\psi\|_{X^{\lf m/2\rf +2}}^{p-1} \|\psi\|_{X^m} \leq \|\psi\|^p_{X^m}.
\end{split}\]
On the other hand, for $r\geq R+2$ we have, using \eqref{ap-S} (with $b=\frac{\ga}{q-1}$), \eqref{ap-Sinfty} (with $b=\ga$) and the fact that $\frac{\ga-1}{q-1}\le -\ga$:
\[\begin{split}
& \Bigl\| Z^{\alpha_1} u \dots Z^{\alpha_p} \psi\Bigr\|_{L^1L^2(r\geq R+2)}  \lesssim \|r^{-\ga}\chi Z^{\alpha_p} \psi\|_{L^q L^q L^2} \|r^{\ga} Z^{\alpha_1} \psi \dots Z^{\alpha_{p-1}} \psi\|_{L^{\frac{q}{q-1}} L^{\frac{2q}{q-2}} L^\infty(r\geq R+2)} \\ & \lesssim  \|r^{-\ga}\chi Z^{\alpha_p} \psi\|_{L^q L^q L^2} \|\psi_{\leq |\alpha_{p-1}|}\|^{p-q}_{L^\infty L^\infty(r\geq R+2)} \|r^{\ga} (\psi_{\leq |\alpha_{p-1}|})^{q-1} \|_{L^{\frac{q}{q-1}} L^{\frac{2q}{q-2}} L^\infty(r\geq R+2)} \\ & \lesssim \|\psi\|_{X^m} \|\psi\|^{p-q}_{X^{\lf m/2\rf +6}} \|r^{\frac{\ga}{q-1}} \psi_{\leq |\alpha_{p-1}|} \|^{q-1}_{L^q L^{\frac{2q(q-1)}{q-2}} L^\infty(r\geq R+2)}  \lesssim \|\psi\|^{1+p-q}_{X^m}\|r^{\frac{\ga-1}{q-1}} \psi_{\leq |\alpha_{p-1}|}\|^{q-1}_{L^q L^q L^\infty(r\geq R+1)}\\ & \lesssim \|\psi\|^{1+p-q}_{X^m}  \|r^{-\ga} \chi \psi_{\leq |\alpha_{p-1}|+6}\|^{q-1}_{L^q L^q L^2} \lesssim \|\psi\|^{1+p-q}_{X^m}  \|\psi\|^{q-1}_{X^{\lf m/2\rf +2}} \lesssim \|\psi\|^p_{X^m}.
\end{split}\]

Finally,
\[\begin{split}
& \Bigl\| r^{-\ga q} \chi^q Z^{\alpha_1} \psi \dots Z^{\alpha_p} \psi \Bigr\|_{L^1L^1L^2}  \lesssim  \|r^{-\ga}\chi Z^{\alpha_p} \psi\|_{L^q L^q L^2} \|r^{-\ga(q-1)} \chi^{q-1}Z^{\alpha_1} \psi \dots Z^{\alpha_{p-1}} \psi\|_{L^{\frac{q}{q-1}} L^{\frac{q}{q-1}} L^\infty} \\ & \lesssim \|r^{-\ga}\chi Z^{\alpha_p} \psi\|_{L^q L^q L^2} \|\psi_{\leq |\alpha_{p-1}|}\|^{p-q}_{L^\infty L^\infty} \|r^{-\ga(q-1)} \chi^{q-1}(\psi_{\leq |\alpha_{p-1}|})^{q-1} \|_{L^{\frac{q}{q-1}} L^{\frac{q}{q-1}} L^\infty} \\ & \lesssim \|\psi\|_{X^m} \|\psi\|^{p-q}_{X^{\lf m/2\rf +2}} \|r^{-\ga} \chi \psi_{\leq |\alpha_{p-1}|} \|^{q-1}_{L^q L^q L^\infty} \lesssim  \|\psi\|^p_{X^m}.
\end{split}\]

The proof of \eqref{nlnbd} is now complete.

\end{proof}

\newsection{$r^p$ estimates}

The local energy spaces from the previous section are enough to obtain a weak decay estimate, see Theorem 6.1 from \cite{LT}.

\begin{theorem}\label{ptwsedcy1} (Theorem 6.1, \cite{LT})

Let $T$ be a fixed time. We then have for $T\leq \tt\leq 2T$:
\begin{equation}\label{ptdecayu}
|\phi_{\leq _{|\alpha|}}| \leq C_{m} \la \tt\ra^{-1} \la \tt-\rs\ra^{1/2} \|\phi_{\leq |\alpha|+13}\|_{LE^1[T, 2T]}.
\end{equation}
\end{theorem}

In particular, this implies 
\beq\label{ptwse1}
|\phi_{\leq |\alpha|}| \lesssim \varepsilon\la \tt\ra^{-1} \la \tt-\rs\ra^{1/2} 
\eq
for all $|\alpha|\leq N-19$.

One idea to continue here is to rewrite the equation as 
\[
\Box \phi = (\Box_\phi - \Box_K) \phi \pm \phi^p,
\]
and use the fundamental solution for the Minkowski, combined with \eqref{ptwse1}, to improve the rate of decay. This works for $p>5$, but not for smaller $p$. Instead we first use $r^p$ estimates inspired by the work of Dafermos-Rodnianski \cite{DR} to improve the pointwise decay.

Let $\pa_v = \pa_t + \pa_\rs$, and $\ang$ denote angular derivatives. We introduce the weighted local energy norm $LE_\ga$ for $\ga>0$: 
\begin{equation}
 \| \phi\|_{LE_\ga} =   \|  r^{\frac{\ga-1}2} \phi\|_{L^2(\M)} , 
\label{ledefp}\end{equation}
and its $H^1$ counterpart
\begin{equation}
  \| \phi\|_{LE_\ga^1} = \| \pa_v \phi\|_{LE_\ga} + \| \ang \phi\|_{LE_\ga} + \|  r^{-1} \phi\|_{LE_\ga}. 
\end{equation}
We also define the degenerate norm
\[
  \| \phi\|_{LE^1_{w,\ga}} = \| (1-\chi_{ps}) \pa_v \phi\|_{LE_\ga} + \| (1-\chi_{ps}) \ang \phi\|_{LE_\ga} + \|  r^{-1} \phi\|_{LE_\ga},
\]
and the weighted energy 
\[
E_\ga[\phi](\tt) = \int_{\Sigma(\tt)} r^\ga \left(|\pa_v \phi|^2 + |\ang\phi|^2 + r^{-2} \phi^2\right) d\Sigma_K(\tt).
\]

For the dual norm, we define
\begin{equation}
 \| f\|_{LE_{\ga}^*} = \| r^{\frac{\ga}2} f\|_{LE^*}. 
\label{lesdefp}\end{equation}

We will prove the following linear estimate:
\begin{theorem}\label{LEplin}
Assume that $\phi$ solves
\[
\Box_K \phi = F, \qquad \phi |_{\at=0} = \phi_0, \qquad \tilde T \phi |_{\at=0} = \phi_1\ .
\]
Then for any $0<\ga<2$ and compactly supported $(\phi_0, \phi_1)$ we have
\beq\label{rplin}
\sup _{\tt} E_\ga[\phi](\tt) + \| \phi\|_{LE^1_{w,\ga}}^2 \lesssim E[\phi](0) + E[\phi](T) + \| \phi\|^2_{LE^1_{w}[0,T]} + \|F\|^2_{LE_{\ga}^*}.
\eq
\end{theorem}

A similar result appears in a paper by Stogin \cite{St}.

\begin{proof}

Recall that the energy-momentum tensor is given by
\[
Q_{\alpha\beta}[g]=\partial_\alpha \phi \partial_\beta \phi -
\frac{1}{2}g_{\alpha\beta}\partial^\mu \phi \partial_\mu \phi.
\]

Its contraction with respect to a vector field $X$ is denoted by
\[
P_\alpha[g,X]=Q_{\alpha\beta}[g]X^\beta.
\]
Its divergence is given by
\[
\nabla^\alpha P_\alpha[g,X] = \Box_g \phi \cdot X\phi + \frac{1}{2}Q[g, X], 
\qquad
Q[g, X] = Q_{\alpha\beta}[g]\pi_X^{\alpha \beta},
\]
where $\pi_X^{\alpha \beta} $ is the deformation tensor of $X$, which is given in terms of the Lie derivative by
\[
\pi_{\alpha\beta}^X=\nabla_\alpha X_\beta + \nabla_\beta X_\alpha=({\mathcal L}_X g)_{\alpha\beta}.
\]

In coordinates, one can write
\begin{equation}\label{defform}
Q[g, X] =-\frac{1}{\sqrt{|g|}}(X(\sqrt{|g|}g^{\alpha\beta}))
\partial_\alpha \phi\,\partial_\beta \phi \\
+(g^{\alpha\gamma}\partial_\gamma X^\beta +g^{\beta\gamma}\partial_\gamma X^\alpha)\partial_\alpha \phi\partial_\beta \phi- \partial_\gamma X^\gamma \, g^{\alpha\beta}\partial_\alpha \phi\,\partial_\beta \phi.
\end{equation}

More generally, for a vector field
$X$, a scalar function $q$ and a one-form $m$, we define
\[
P_\alpha[g, X, q, m] = P_\alpha[g,X] + q \phi \partial_\alpha \phi - \frac12
(\partial_\alpha q) \phi^2 + \frac{1}{2}m_{\alpha}\phi^2.
\]
The divergence of $P$ is
\begin{equation}
\nabla^\alpha P_\alpha[g,X,q, m] =  \Box_g \phi \Bigl(X\phi +
 q u\Bigr)+ Q[g,X,q, m],
\label{div}\end{equation}
where
\[
 Q[g,X,q, m] =
\frac{1}{2}Q[g, X] + q
\partial^\alpha \phi\, \partial_\alpha \phi  + m_\alpha \phi\,
\partial^\alpha \phi + (\nabla^\alpha m_\alpha -\frac{1}{2}
\nabla^\alpha \partial_\alpha q) \, \phi^2.
\]

Let $\M_{[0, T]} = \M \cap [0, T]\times \R^3$. The divergence theorem yields, assuming that $X$, $q$ and $m$ are supported in $\{r>4M\}$, that
\begin{equation}
\int_{\M_{[0, T]}} Q[g, X, q, m]  dV_g =
- \int_{\M_{[0, T]}} \Box_{K} \phi \left(X\phi + q \phi\right) dV_g
+ BDR^g,
\label{intdiv}\end{equation}
where $BDR^g$ denotes the boundary terms
\[
BDR^g = \int_{\Sigma(\tt)} \langle d\tt, P[g,X ,q, m]\rangle
  d\Sigma \Big|_{0}^{T}.
\]

Even though one can do the following computations in Kerr, it is easier to first do the Schwarzschild case, and treat Kerr perturbatively.

 Fix $\gamma<2$ and pick $0 < \delta$ small so that
\beq\label{delta}
(1-\delta)^2 - 2(1-\ga\delta) < 0.
\eq 
Let
\[
 X =  r^\ga \pa_v   ,\qquad q(r) =  r^{\gamma-1}\weight, \qquad m = \ga(1-\delta) r^{\gamma-2} dv.
\]

 We compute, using the fact that $\pa_{\rs} = \weight\pa_r$
\[\begin{split}
& P_0[g_S, X, q, m]  =  \frac{r^\ga}2 \left(|\pa_t \phi|^2 + |\pa_\rs \phi|^2 + \weight|\ang \phi|^2 + 2\pa_t\phi \pa_\rs \phi\right) + r^{\ga-1}\weight \phi\pa_t\phi +\frac{\ga(1-\delta)r^{\ga-2}}2\phi^2
\\ & = \frac{r^\ga}2 \left(|\pa_v \phi|^2 + \frac1r \weight\phi\right)^2 + \frac{r^\ga}2 \weight|\ang \phi|^2 - r^{\ga-1}\weight^2\left(\phi \pa_r\phi+ \frac{1}{2r}\phi^2\right)
+\frac{\ga(1-\delta)r^{\ga-2}}2\phi^2.
\end{split}\]

Recall that for the Schwarzschild metric $dV_S = r^2 d\tt drd\omega$, and $d\Sigma_S = r^2 drd\omega$. Let $R_2$ be large enough, and $\chi_{R_2}(r)$ be a smooth cutoff equal to $1$ when $r\geq R_2$ and supported in $\{r\geq \frac{R_2}2\}$. After integrating by parts we obtain
\[\begin{split}
& -\int_{\Sigma(\tt)} \la d\tt, P[g_S, \chi_{R_2} X, \chi_{R_2} q, \chi_{R_2} m]\ra d\Sigma_S = \int_{\Sigma(\tt)} \chi_{R_2}\weight^{-1} \left[\frac{r^\ga}2\left(\pa_v \phi + \frac1r \weight\phi\right)^2 \right. \\ & +  \frac{r^{\ga}}{2}  \weight |\ang \phi|^2 - r^{\ga-1}\weight^2\left(\phi \pa_r\phi+ \frac{1}{2r}\phi^2\right) + \frac{\ga(1-\delta)r^{\ga-2}}2\phi^2\Bigl.\Bigr] d\Sigma_S= \\ & \int_{\Sigma(\tt)} \chi_{R_2} \frac{r^\ga}2\left[\weight^{-1}\left(\pa_t \phi + \pa_\rs \phi + \frac1r \weight\phi\right)^2 + |\ang \phi|^2 + \left(\ga(2-\delta) -\frac{2M\ga}{r}\right)\frac{\phi^2}{r^2} \right] \\& + \chi_{R_2}^\prime \frac{r^{\ga-1}-2Mr^{\ga-2}}2 \phi^2 d\Sigma_S.
\end{split}\]

Since the support of $\chi_{R_2}^\prime$ is contained in $\{\frac{R_2}2\leq r\leq R_2\}$, we have by Hardy's inequality that
\[
\int_{\Sigma(\tt)} \chi_{R_2}^\prime \frac{r^{\ga-1}-2Mr^{\ga-2}}2 \phi^2 d\Sigma_S \lesssim E[\phi](\tt).
\]

Moreover, due to \eqref{KSdiff} we see that
\[
Q_{\alpha\beta}[g_K] - Q_{\alpha\beta}[g_S] \lesssim \frac{1}{r^2} |\pa\phi|^2, 
\]
which immediately implies, since $\gamma<2$, that
\[
\la d\tt, P[g_S, \chi_{R_2} X, \chi_{R_2} q, \chi_{R_2} m]\ra - \la d\tt, P[g_K, \chi_{R_2} X, \chi_{R_2} q, \chi_{R_2} m]\ra \lesssim r^{\gamma-2} |\pa\phi|^2 \lesssim |\pa\phi|^2.
\]

Finally, we have that $\sqrt{|g_K|}\approx \sqrt{|g_S|}$. 

We thus have 
\beq\label{bdryest}
E_\ga[\phi](\tt) \lesssim -\int_{\Sigma(\tt)}  \la d\tt, P[g_K, \chi_{R_2} X, \chi_{R_2} q, \chi_{R_2} m]\ra  d\Sigma_K + E[\phi](\tt). 
\eq

Moreover, since the initial data has compact support, we clearly have that when $\tt=0$:
\beq\label{inbdryest}
-\int_{\Sigma(0)}  \la d\tt, P[g_K, \chi_{R_2} X, \chi_{R_2} q, \chi_{R_2} m]\ra  d\Sigma_K \approx E[\phi](0).
\eq

We now compute the spacetime term. Using \eqref{defform} we have for any $f(r)$:
\[
Q[g_S, f(r)\pa_t] = 2f^{\prime} \weight\pa_t\phi\pa_r\phi
\]
\[\begin{split}
Q\left[g_S, \weight f(r)\pa_r\right] = & f^{\prime}(\pa_t\phi)^2 + f^{\prime}\weight^2(\pa_r\phi)^2 + \left[ 2f \left(1-\frac{3M}r\right) -\weight f^\prime\right] |\ang\phi|^2 \\ & -\frac{1}{2r} \weight f \pa^\ga \phi \pa_\ga \phi
\end{split}\]
and thus
\[
Q[g_S, X, q, 0] = \frac{\ga r^{\ga-1}}2 (\pa_v \phi)^2 + r^{\ga-1} \left(\frac{2-\ga}2+ \frac{(\ga-3)M}{r}\right)|\ang\phi|^2 - \frac12 \Box_{g_S} q.
\]
We also compute
\[
- \frac12 \Box_{g_S} q = \ga(1-\ga)r^{\ga-3}\left(1+O(\frac1r)\right),
\]
which unfortunately has the wrong sign when $\ga>1$. On the other hand, we have
\[
\nabla^\alpha m_\alpha = (1-\delta)\ga^2 r^{\ga-3}\left(1+O(\frac1r)\right)
\]
and thus
\[
\nabla^\alpha m_\alpha - \frac12 \Box_{g_S} q = \ga(1-\delta\ga)r^{\ga-3}\left(1+O(\frac1r)\right). 
\]

Due to \eqref{delta} we see that
\[
\frac{\ga r^{\ga-1}}2 (\pa_v \phi)^2 + \ga(1-\delta)r^{\ga-2} \phi\pa_v\phi + \ga(1-\delta\ga)r^{\ga-3}\phi^2 \gtrsim r^{\ga-1} (\pa_v \phi)^2 + r^{\ga-3} \phi^2
\]
and thus for $R_2$ large enough and $r\geq R_2$ we have
\[
Q[g_S, X, q, m] \gtrsim r^{\ga-1} \left(|\pa_v \phi|^2 + |\ang\phi|^2\right) + r^{\ga-3} \phi^2.
\]

Due to the support properties of $\chi_{R_2}$ we get
\beq\label{bulk}
\int_{\M_{[0, T]}} Q[g_S, \chi_{R_2} X, \chi_{R_2} q, \chi_{R_2} m] dV_K \gtrsim \| \phi\|^2_{LE^1_{w,\ga}[0, T]} - \| \phi\|^2_{LE^1_{w}[0,T]}. 
\eq

Using \eqref{KSdiff}  we see that 
\beq\label{intest}
Q[g_S, \chi_{R_2} X, \chi_{R_2} q, \chi_{R_2} m] - Q[g_K, \chi_{R_2} X, \chi_{R_2} q, \chi_{R_2} m] \lesssim \chi_{R_2/2} \left(r^{\ga-3} |\pa\phi|^2 + r^{\ga-4} |\phi|^2\right).
\eq

We thus obtain from \eqref{intdiv}, \eqref{bdryest}, \eqref{inbdryest}, \eqref{bulk} and \eqref{intest} that
\[
E_\ga[\phi](T) + \| \phi\|^2_{LE^1_{w,\ga}[0, T]} \lesssim \int_{\M_{[0, T]}} \left| (\Box_{K} \phi) \chi_{R_2} \left(X\phi + q \phi\right)\right| dV_K + E[\phi](0) + E[\phi](T) + \| \phi\|^2_{LE^1_{w}[0,T]}. 
\]

The result \eqref{rplin} now follows by Cauchy Schwarz.

\end{proof}

\newsection{Setup for pointwise estimates}

Note first that we can consider, instead of $\Box_K$, an operator that looks like $\Box$ (with respect to $\rs$) plus a long range perturbation; this will allow us to directly apply the results in  \cite{MTT}. Indeed, let 
\[
P = |g_K|^{1/4}(-g_K^{\tt\tt})^{-1/2}\Box_K (-g_K^{\tt\tt})^{-1/2} |g_K|^{-1/4}.
\]

$P$ is self-adjoint with respect to $d\tt dx$. More importantly, a quick computation yields that 
\[
P = \pa_{\alpha} \left(g_K^{\alpha\beta} (-g_K^{\tt\tt}) \partial_{\beta}\right) + V, \quad V = |g_K|^{1/4}(-g_K^{\tt\tt})^{-1/2} \Box_K \left((-g_K^{\tt\tt})^{-1/2} |g_K|^{-1/4}\right).
\]

It is easy to see that $V\in S^Z(r^{-3})$. Moreover, in Schwarzschild we have that for large $r$, $-g_S^{\tt\tt} = g_S^{r^*r^*}$ and $g_S^{\tt r^*} = 0$. We thus have
\[
P = \Box + P_{lr},  
\]
where the long range spherically symmetric part $P_{lr}$ has the form
\beq\label{Plr}
P_{lr} = g_{lr}(r)\Delta_{\omega} + V, \qquad g_{lr} \in S_{rad}^Z(r^{-3}), \qquad V \in S^Z(r^{-3}).
\eq

For the Kerr metric, using \eqref{KSdiff} yields
\beq\label{Pdec}
P = \Box + P_{lr} + P_{sr},
\eq
where the short-range part $P_{sr}$ has the form
\beq\label{Psr}
P_{sr} = \partial_\alpha g_{sr}^{\alpha \beta}\partial_\beta, \quad g_{sr}^{\alpha \beta} \in S^Z_{der}(r^{-2}).
\eq

We now see that $\phi$ satisfies
\[
P \phi = (-g_K^{\tt\tt}) \phi^5 + h_1 \phi + h_2  \pa\phi, \quad h_1\in S^Z(r^{-3}), \quad h_2 \in S^Z_{der}(r^{-2}).
\]
Now pick any $|\alpha|\leq N$. After commuting with vector fields, using \eqref{Pdec}, \eqref{Plr}, and \eqref{Psr}, we obtain
\beq\label{Peq}
P \phi_{\alpha} = F_{\alpha} + G_{\alpha},  \qquad \phi_{\alpha} |_{\at=0} = \phi_0^{\alpha}, \qquad \tilde T \phi_{\alpha} |_{\at=0} = \phi_1^{\alpha},
\eq
with
\beq\label{Fterm}
F_{\alpha} = \left((-g_K^{\tt\tt}) \phi^5\right)_{\alpha} \in S^Z(1) \sum_{\substack{|\alpha_1|\leq \dots\leq|\alpha_p|\leq |\alpha| \\ \alpha_1 +\dots\alpha_p=\alpha}}\prod_{j=1}^p \phi_{\alpha_j},
\eq
\beq\label{Gterm}
G_{\alpha} \in S^Z(r^{-3}) \phi_{\leq |\alpha|+6} + S^Z_{der}(r^{-2}) \pa \phi_{\leq |\alpha|+5},
\eq
and
\beq\label{indata}
\|\phi_0^{\alpha}\|_{H^1} + \|\phi_1^{\alpha}\|_{L^2} \lesssim \varepsilon.
\eq

We will use \eqref{Peq} to control the solution when $\rs$ is small.

On the other hand, for large $\rs$ it is more convenient to work with $\Box$ and treat the rest perturbatively. Let $\chi_{out}$ be a cutoff equal to $1$ for large $r$, and $\psi_{\alpha} = \chi_{out} \phi_{\alpha}$. Then $\psi_{\alpha}$ satisfies 
\beq\label{Boxeq}
\Box \psi_{\alpha} = F_{\alpha} + G_{\alpha},  \qquad \psi_{\alpha} |_{\at=0} = \tilde T \psi_{\alpha} |_{\at=0} = 0,
\eq
where $F_{\alpha}$ and $G_{\alpha}$ are supported away from $0$ and satisfy \eqref{Fterm}, \eqref{Gterm}.

We will decompose $\psi_{\alpha}$ as
\[
\psi_{\alpha} = \psi_1 + \psi_2,
\]
where 
\beq\label{psi1def}
\Box \psi_1 = G_{\alpha}, \qquad \psi_1 |_{\at=0} = 0, \qquad \pa_t \psi_1 |_{\at=0} = 0,
\eq

\beq\label{psi2def}
\Box \psi_2 = F_{\alpha}, \qquad \psi_1 |_{\at=0} = 0, \qquad \pa_t \psi_1 |_{\at=0} = 0.
\eq

By finite speed of propagation, $F_{\alpha}$ and $G_{\alpha}$ are supported in the forward light cone $\{|x| < \tt+CR_1\}$. By a time translation, we may assume that $F_{\alpha}$ and $G_{\alpha}$ are supported in the forward light cone $\{|x| < \tt\}$.

Finally, in the next sections $n$ will represent a  large constant,
which does not depend on $\alpha$, but may increase from one estimate to the next.

\medskip

\newsection{Estimates for the fundamental solution}

The goal of this section is to prove pointwise estimates for solutions to the inhomogeneous wave equation on Minkowski backgrounds.

We will first prove the following lemma, which gives pointwise bounds for the solution assuming the inhomogeneity lies in certain weighted $L^\infty$ spaces.

For any $\beta, \gamma, \eta\in\R$, we define the weighted $L^\infty$ norms
\beq\label{Linftywght}
\|G\|_{L_{\beta,\gamma,\eta}^{\infty}} = \| \la r\ra^{\beta} \la t\ra^{\gamma} \la t-r\ra^{\eta}  H(t, r)\|_{L_{t,r}^{\infty}}, \quad H(t,r)= \sum_0^2 \|\Omega^i G (t, r\omega)\|_{L^2(\S^2)}.
\eq
 
\begin{lemma}\label{Minkdcy}
i) Let $\psi$ solve 
\beq\label{Mink1}
\Box \psi = G , \qquad \psi(0) = 0, \quad \pa_t \psi(0) = 0,
\eq
where $G$ is supported in $\{|x| \leq t\}$. For any $1<\beta\leq 3$ and $\eta\neq 1$ we have:

\beq\label{lindcy1}
\psi(t, x)\lesssim \frac{1}{\la r\ra\la t-r\ra^{\beta+\tilde\eta}} \|G\|_{L_{\beta,1,\eta}^{\infty}}, 
\eq

where we define, for any arbitrary $\delta>0$,
\[
\tilde\eta = \left\{ \begin{array}{cc} \eta-\delta -2& \eta<1 ,  \cr -1 & \eta > 1 
  \end{array} \right. .
\]

ii) Let $\psi$ solve 
\beq\label{Mink2}
\Box \psi = \pa_t \tG, \qquad \psi(0) = 0, \quad \pa_t \psi(0) = 0,
\eq
where $\tG$ is supported in $\{t/2 \leq |x| \leq t\}$. For any $2<\beta\leq 3$ and $\eta\neq 1$ we have:

\beq\label{lindcy2}
\psi(t, x)\lesssim \frac{1}{\la r\ra\la t-r\ra^{1+\tilde\eta}}\left(\|\tG\|_{L_{\beta-1,1,\eta}^{\infty}} + \|S\tG\|_{L_{\beta-1,1,\eta}^{\infty}} + \|\Omega\tG\|_{L_{\beta-1,1,\eta}^{\infty}}+ \|\la t-r\ra \pa \tG|\|_{L_{\beta-1,1,\eta}^{\infty}} \right).
\eq

\end{lemma}

We remark here that \eqref{lindcy1} is similar to previous classical results, see for instance \cite{John}, \cite{As}, \cite{STz}, \cite{Sz2}.
\begin{proof}

We use the ideas from \cite{MTT}. Let us first prove \eqref{lindcy1}. Define
\begin{equation}\label{Hdef}
  H(t,r)= \sum_0^2 \|\Omega^i G (t, r\omega)\|_{L^2(\S^2)}.
\end{equation}
By Sobolev embeddings on the sphere, we have $|G|\lesssim H$. Let $v$ be the radial solution to
\begin{equation}\label{1dbox}
  \Box v = H, \qquad v[0]=0.
\end{equation}

By the positivity of the fundamental solution, we have that $|\psi| \lesssim |v|$. On the other hand, we can write $v$ explicitly:
\begin{equation}\label{Hsol}
rv(t,r) = \frac12 \int_{D_{tr}} \rho H(s,\rho) ds d\rho,
\end{equation}
where $D_{tr}$ is the rectangle 
\[
D_{tr}=\{ 0 \leq s - \rho \leq t-r, \quad  t-r \leq s+\rho \leq t+r 
 \}.
\]

We partition the set $D_{tr}$ into a double dyadic manner
as 
\[
D_{tr} = \bigcup_{R \leq t}  D_{tr}^R, \quad D_{tr}^R = D_{tr}\cap \{R<r<2R\}
\]
and estimate the corresponding parts of the above integral.

We clearly have
\[
\int_{D_{tr}^R} \rho H ds d\rho \lesssim \|G\|_{L_{\beta, 1,\eta}^{\infty}} \int_{D_{tr}^R} \rho^{1-\beta} \la s\ra^{-1} \la s-\rho\ra^{-\eta} d\rho ds. 
\]

We now consider two cases:

(i) $R < (t-r)/8$. Here we have $\rho \sim R$ and $s\approx s-\rho \approx \la t-r\ra$;
therefore we obtain

\[
\int_{D_{tr}^R} \rho^{1-\beta} \la s\ra^{-1} \la s-\rho\ra^{-\eta} d\rho ds  \lesssim  R^{3-\beta} \la t-r\ra^{-1-\eta}, 
\]
and after summation, using that $\beta\leq 3$, we obtain
\beq\label{ptwsecpt}
\sum_{R < (t-r)/8} \int_{D_{tr}^R} \rho H ds d\rho \lesssim \frac{\ln\la t-r\ra \la t-r\ra^{3-\beta} }{\la t-r\ra^{1+\eta}} \lesssim \frac{1}{\la t-r\ra^{\beta+\tilde\eta}}.
\eq

(ii) $(t-r)/ 8 < R < t$. Here we have $\rho \sim R$ and $s\gtrsim R$. Denote $u=s-\rho$; then
\[
\int_{D_{tr}^R} \rho^{1-\beta} \la s\ra^{-1} \la s-\rho\ra^{-\eta} d\rho ds  \lesssim R^{1-\beta} \int_{0}^{t-r} \la u\ra^{-\eta} du \lesssim R^{1-\beta} \la t-r\ra^{\mu(\eta)},
\]
where 
\[
\mu(\eta)= \left\{ \begin{array}{cc} 1-\eta& \eta<1 ,  \cr 0 & \eta > 1 
  \end{array} \right. .
\]
Since $\beta>1$, we obtain after summation
\beq\label{ptwsefar}
\sum_{R > (t-r)/8} \int_{D_{tr}^R} \rho H ds d\rho \lesssim \la t-r\ra^{1-\beta+\mu(\eta)}.
\eq

The conclusion \eqref{lindcy1} follows from \eqref{ptwsecpt} and \eqref{ptwsefar}.

We now prove \eqref{lindcy2}. Let $\tpsi$ be the solution to 
\beq\label{Mink}
\Box \tpsi = \tG, \qquad \tpsi[0] = 0.
\eq

Clearly $\psi = \pa_t \tpsi$. We also note that in the support of $\tG$ we have
\[
(t \partial_i + x_i \partial_t) \tG \lesssim |S\tG| + |\Omega \tG| + \la t-r\ra |\pa_r \tG|.
\]

By \eqref{lindcy1}, applied to $\tpsi$, $\nabla \tpsi$,  
$\Omega \tpsi$, $S \tpsi$ and $(t \partial_i
+ x_i \partial_t) \tpsi$ we obtain 
\[
|\tpsi|+|\nabla \tpsi| +|S\tpsi|+|\Omega \tpsi|  + \sum_i  | (t \partial_i
+ x_i \partial_t) \tpsi|  \lesssim \frac{1}{\la r\ra\la t-r\ra^{\beta+\tilde\eta-3}} \|\tG\|_{L_{\beta-1,1,\eta}^{\infty}}. 
\]
The above left hand side dominates $\la t-r\ra \partial_t \tpsi$;
therefore the proof of the lemma is complete.

\end{proof}

We will also use the following lemma, which gives pointwise control of $\psi_2$ by the $r^p$ norms of Section 4.
\begin{lemma}\label{ptwsedcyrp}
For some $\alpha$, let $\psi_2$ solve \eqref{psi2def}, and assume also that \eqref{ptwse1} holds. We then have for any $\ga<2$, that
\beq\label{nonlindcy}
\psi_2(\tt, x) \lesssim \frac{1}{\la r\ra (t-\rs)^{\ga-\frac32}}\left\|r^{-1}\phi_{\leq |\alpha|+ 6}\right\|_{LE_\ga}^2.
\eq
\end{lemma}

\begin{proof}

The proof is similar to that of Lemma~\ref{Minkdcy}. Define
\begin{equation}\label{Hdef2}
  H(\tt,\rs)= \|\phi_{\leq |\alpha|}\|^{p-2}_{L^{\infty}(\S^2)} \|\phi_{\leq |\alpha|+6}\|_{L^2(\S^2)}^2.
\end{equation}
By Sobolev embeddings on the sphere, we have $|F_{\alpha}|\lesssim H$. Let $v$ be the radial solution to
\[
  \Box v = H, \qquad v[0]=0.
\]

By the positivity of the fundamental solution, we have that $|\phi_2| \lesssim |v|$. On the other hand, we can write $v$ explicitly as in \eqref{Hsol}:
\[
rv(t,r) = \frac12 \int_{D_{tr}} \rho H(s,\rho) ds d\rho
\]
where $D_{tr}$ is the rectangle 
\[
D_{tr}=\{ 0 \leq s - \rho \leq \tt-\rs, \quad  \tt-\rs \leq s+\rho \leq \tt+\rs 
 \}.
\]

Moreover, due to \eqref{ptwse1}, and the fact that $p\geq 3$, we see that
\beq\label{Hbd}
H(\tt, \rs) \lesssim \frac{\la \tt-\rs\ra^{1/2}}{\la \tt\ra} \|\phi_{\leq |\alpha|+6}\|_{L^2(\S^2)}^2.
\eq

We again consider two cases:

(i) $R < (\tt-\rs)/8$. Since $\rho, \la s-\rho \ra\lesssim \la s\ra$ and $s \gtrsim \tt-\rs $, we have
\[\begin{split}
\int_{D_{tr}^R} \rho H ds d\rho & \lesssim \frac{1}{(\tt-\rs)^{1/2}} \int_{D_{tr}^R} \rho\|\phi_{\leq |\alpha|+6}\|_{L^2(\S^2)}^2 d\rho ds \lesssim \frac{R^{2-\ga}}{(\tt-\rs)^{1/2}} \int_{D_{tr}^R} \rho^{\ga-3} \|\phi_{\leq |\alpha|+6}\|_{L^2(\S^2)}^2 \rho^2 d\rho ds \\& \lesssim \frac{R^{2-\ga}}{(\tt-\rs)^{1/2}} \left\|r^{-1}\phi_{\leq |\alpha|+ 6}\right\|_{LE_\ga}^2.
\end{split}\]

(ii) $(\tt-\rs)/ 8 < R < \tt$. In this case, since $\tt-\rs \lesssim\rho\leq s$, we have for  small $0<\delta$ that
\[
\frac{\rho^{2-\ga}\la s-\rho\ra^{1/2}}{s} \lesssim \frac{R^{-\delta}}{(\tt-\rs)^{\ga-3/2-\delta}},
\]
and thus 
\[
\int_{D_{tr}^R} \rho H ds d\rho \lesssim \int_{D_{tr}^R} \frac{\rho^{2-\ga}\la s-\rho\ra^{1/2}}{s} \rho^{\ga-3} \|\phi_{\leq |\alpha|+6}\|_{L^2(\S^2)}^2 \rho^2 d\rho ds \lesssim \frac{R^{-\delta}}{(\tt-\rs)^{\ga-3/2-\delta}} \left\|r^{-1}\phi_{\leq |\alpha|+ 6}\right\|_{LE_\ga}^2.
\]

The conclusion now follows after summing over $R$. 

Finally, in the next section we will use the following result, see Lemma 3.10 from \cite{MTT}, to improve the bounds for $\psi_1$:
\begin{lemma}\label{MinkdcyLE}
Assume that $\psi$ solves \eqref{Mink1}. The following estimate holds for large enough $n$:
\begin{equation}\label{firstdecestderiv}
  \psi(t, x)\lesssim \frac{\log \la t-r\ra}{\la r\ra \la t-r\ra^{\frac{1}{2}}}
 \|r G_{\leq n}\|_{LE^*}.
\end{equation}
\end{lemma}

\end{proof}

\medskip

\newsection{An improved bound}

We now pick $n$ suitably large, and assume that $m+n \ll N$. The main goal of the section is to prove that, for any $0< q<\frac12$, we have
\beq\label{ptwsein}
\phi_{\leq m+n} \lesssim \varepsilon\la r\ra^{-1} \la \tt-\rs\ra^{-q}.  
\eq

Clearly this holds for $r\lesssim 1$ by \eqref{ptwse1}. For $r$ large, this will follow by a continuity argument. We will assume that the following a-priori bounds hold for some large constant $\tilde C$ independent of $\varepsilon$ and $\tt$:
\begin{equation}\label{enapbds}
|\phi_{\leq m+n}| \leq \tilde C \varepsilon\la r\ra^{-1} \la \tt-\rs\ra^{-q}. 
\end{equation}

Clearly such a bound holds for small $\tt$ by Sobolev embeddings and the compact support of the initial data. We now assume that \eqref{enapbds} holds for all $0\leq \tt\leq T$, and we will improve the constant on the right hand side by a factor of $\frac12$. For the rest of the section, the implicit constants will not depend on $\tilde C$.

We will now show, that under assumption \eqref{enapbds}, we have that the solution to \eqref{nlw} satisfies, for small enough $\varepsilon$,
\beq\label{rpnonlin}
\sup _{0\leq\tt\leq T} E_\ga[\phi_{\leq m+n+6}](\tt) + \| \phi_{\leq m+n+6}\|^2_{LE_\ga^1[0, T]} \lesssim \varepsilon^2.
\eq

Indeed, pick any $|\alpha|\leq m+n+6$ and apply Theorem~\ref{LEplin} with $\ga= \frac32 + q$ to $\phi_{\alpha}$. We obtain, using also \eqref{comm}
\beq\label{rpcomm}
\begin{split}
& \sup _{0\leq\tt\leq T} E_\ga[\phi_{\alpha}](\tt) + \| \phi_{\alpha}\|^2_{LE_\ga^1[0, T]}  \lesssim   E_[\phi_{\alpha}](0) + E[\phi_{\alpha}](T) + \| \phi_{\alpha}\|^2_{LE^1_{w}[0,T]} \\ & + \|(\phi^p)_{\alpha}\|_{LE_{\ga}^*[0,T]}^2 + \|r^{-2+} \pa\phi_{\leq |\alpha|+1}\|^2_{LE_{\ga}^*[0,T]}.
\end{split}\eq
By \eqref{derivloss} we have that
\[
E[\phi_{\alpha}](T) + \| \phi_{\alpha}\|^2_{LE^1_{w}[0,T]} \lesssim \varepsilon^2.
\]
Since $\ga<2$, we have 
\beq\label{vfcomp}
\|r^{-2+} \pa\phi_{\leq |\alpha|+1}\|^2_{LE_{\ga}^*[0,T]} \lesssim \|\phi_{\leq m+n+9}\|^2_{LE^1} \lesssim \varepsilon^2.
\eq

Moreover, \eqref{nonlinvf} and \eqref{enapbds} yield
\[\begin{split}
\|(\phi^p)_{\alpha}\|^2_{LE_{\ga}^*[0,T]}& \lesssim (\tilde C \varepsilon)^{2(p-1)} \|\left(\la r\ra^{-1} \la \tt-\rs\ra^{-\ga}\right)^{p-1}\phi_{\leq |\alpha|}\|^2_{LE_{\ga}^*[0,T]} \leq (\tilde C \varepsilon)^{2(p-1)} \|\la r\ra^{-2} \phi_{\leq |\alpha|}\|^2_{LE_{\ga}^*[0,T]} \\ & \leq (\tilde C \varepsilon)^{2(p-1)} \|\phi_{\leq |\alpha|}\|^2_{LE_\ga^1[0, T]}.
\end{split}\]
and this term can be absorbed in the LHS of \eqref{rpcomm} for small enough $\varepsilon$. The conclusion \eqref{rpnonlin} now follows.

In particular, we have that
\beq\label{lotrp}
\| \la r\ra^{-1} \phi_{\leq m+n+6} \|_{LE_\ga} \lesssim \varepsilon.
\eq

We now finish the argument. By Lemma~\ref{MinkdcyLE}, \eqref{Gterm} and \eqref{derivloss}, we obtain
\beq\label{psi1in}
\psi_1 \lesssim \frac{\log \la t-r\ra}{\la r\ra \la t-r\ra^{\frac{1}{2}}}
 \|r (G_{\alpha|})_{\leq n}\|_{LE^*} \lesssim \frac{1}{\la r\ra \la t-r\ra^{q}} \|\phi_{m+n}\|_{LE^1} \lesssim \frac{\varepsilon}{\la r\ra \la t-r\ra^{q}}. 
\eq
Moreover, \eqref{lotrp} and \eqref{nonlindcy} yield
\beq\label{psi2in}
\psi_2 \lesssim \frac{1}{\la r\ra (t-\rs)^{q}}\left\|r^{-1}\phi_{\leq |\alpha|+ 6}\right\|_{LE_\ga}^2 \lesssim \frac{\varepsilon^2}{\la r\ra (t-\rs)^{q}}.
\eq

Then \eqref{psi1in} and \eqref{psi2in} imply that
\[
\psi \lesssim \frac{\varepsilon}{\la r\ra \la t-r\ra^{q}}, 
\]
and the conclusion follows if $\tilde C$ is large enough.

\medskip

\newsection{Improved estimates in the interior and for derivatives }

For the forward cone $C = \{ r \leq t\}\cap\M$ we consider  a dyadic decomposition
 in time into sets
\[
C_{T} = \{ T \leq t \leq 2T, \ \ r \leq t\} \cap\M.
\]
For each $C_T$ we need a further double dyadic decomposition of it
with respect to either the size of $t-r$ or the size of $r$, depending
on whether we are close or far from the cone,
\[
C_{T} = \bigcup_{1\leq  R \leq T/4}  C_{T}^{R}  \cup \bigcup_{1\leq  U < T/4} C_T^{U},
\]
where for $R>4M$ , $U > 1$ we set
\[
 C_{T}^{R} = C_T \cap \{ R < r < 2R \},
\qquad
C_{T}^{U} = C_T \cap \{ U < t-r < 2U\}
\]
while for $R=4M$ and $U= 1$ we have
\[
 C_{T}^{R=4M} = C_T \cap \{ r_e < r < 4M \},
\qquad
C_{T}^{U=1} = C_T \cap \{ 0 < t-r < 2\}.
\]
  By $\tilde C_{T}^{R}$
and $\tilde C_{T}^{U}$ we denote enlargements of these sets in both
space and time on their respective scales. We also define
\[ 
C_{T}^{<T/2} = \bigcup_{R < T/4} C_T^R.
\]
and $\tilde C_{T}^{<T/2}$ an enlargement in space-time on its scale.

We will use Propositions 3.15 and 3.16 from \cite{MTT}, stated below. The role of Proposition 3.15 is twofold: it will allow us to obtain a better bound for the derivative in the region $r<t/2$, and improve the decay estimate of the function in the interior region. The role of Proposition 3.16 is to obtain a better bound for the derivative in the region $r>t/2$.

\begin{proposition}
Assume that $Pu = f$. We have for any $m$ and large enough (but $m$-independent) $n$ that
\begin{equation}
\begin{split}
\| u_{\leq m} \|_{L^\infty ( C_{T}^{<T/2})}
& \  +\| \la r \ra \nabla u_{\leq m} \|_{L^\infty ( C_{T}^{<T/2})}  \lesssim 
T^{-\frac32} \|  u_{\leq m+n} \|_{LE( \tilde C_{T}^{<T/2})} 
\\ & \ + T^{-\frac12} \Bigl(\|f_{\leq m+n}\|_{LE^*( \tilde C_{T}^{<T/2})}+ 
\|\la r \ra^2 \nabla f_{\leq m+n}\|_{LE( \tilde C_{T}^{<T/2})}\Bigr).
\end{split}
\label{l2toli}\end{equation}
\label{p:pointr}
\end{proposition}

\begin{proposition}
We have
\begin{equation}
\begin{split}
U \| \nabla u_{\leq m} \|_{L^\infty ( C_{T}^{U})}
  \lesssim & \  \|  u_{\leq m+n} \|_{L^\infty( \tilde C_{T}^{U})} + 
 T^{-\frac12} U^{\frac12}  \| f_{\leq m+n}\|_{L^2( \tilde C_{T}^{U})}\\ & \ + 
T^{-\frac12} U^{\frac32}  \| \nabla f_{\leq m+n}\|_{L^2( \tilde C_{T}^{U})}.
\end{split}
\label{l2toli:u}\end{equation}
\label{p:pointu}
\end{proposition}

As a quick corollary of \eqref{l2toli}, assume that $\phi$ is the solution to \eqref{nlw}, and moreover satisfies
\[
|\phi_{\leq m+n+1}| \lesssim \frac{1}{\la r\ra\la \tt-\rs\ra^{\gamma}}, \quad \gamma > 0.
\]
Let $\ga_1 = \min\{\ga, p\ga - 1\}$. We then we obtain the improved pointwise bound in $C_{T}^{<T/2}$
\beq\label{rtot}
|\phi_{\leq m}| + | \la r \ra \nabla \phi_{\leq m}|\lesssim \frac{1}{\la \tt\ra\la \tt-\rs\ra^{\ga_1}},
\eq
and the improved bound in $C_{T}^{U}$
\beq\label{dercone}
 |\nabla \phi_{\leq m}| \lesssim \frac{1}{\la \tt\ra\la \tt-\rs\ra^{1+\gamma_1}}.
\eq

Indeed, it is easy to check that
\[
\| r^{-1/2} \phi_{\leq m+n}\|_{L^2(\tilde C_T^R)} \lesssim T^{1/2-\gamma}.
\]
Moreover, using the fact that in $\tilde C_{T}^{<T/2}$ we have
\[
\phi^p_{\leq m+n+1} \lesssim \frac{1}{T^{p\gamma}\la r\ra^p}
\]
and furthermore using that $p\geq 3$:
\[
\|(\phi^p)_{\leq m+n}\|_{LE^*( \tilde C_{T}^{<T/2})} + \|\la r \ra^2 \nabla ((\phi^p)_{\leq m+n})\|_{LE( \tilde C_{T}^{<T/2})} \lesssim T^{\frac12-p\gamma}.
\]
This finishes the proof of \eqref{rtot}.

In order to prove \eqref{dercone}, we see that, since in $\tilde C_{T}^{U}$ we have
\[
\phi^p_{\leq m+n+1} \lesssim \frac{1}{U^{p\gamma} T^p}.
\]
Since $U\lesssim T$ and $p\geq 3$, we obtain:
\[
 T^{-\frac12} U^{\frac12}  \| (\phi^p)_{\leq m+n}\|_{L^2( \tilde C_{T}^{U})} + 
T^{-\frac12} U^{\frac32}  \| \nabla (\phi^p)_{\leq m+n}\|_{L^2( \tilde C_{T}^{U})} \lesssim T^{1-p} U^{2-p\ga} \lesssim T^{-1} U^{4-p-p\ga} \lesssim T^{-1} U^{1-p\ga}. 
\]

\medskip

\newsection{The bootstrap argument}

We will now finish the proof. We assume that the bound \eqref{ptwsein} holds for some $q$ close to $\frac12$, and improve it to \eqref{mainbd}.

 Note first that \eqref{ptwsein} already gives the desired rate near the cone, when $\tt-\rs\approx 1$, and thus we can assume from now on that $\tt-\rs > 1$.

Let $\chi_0$ be a smooth cutoff supported in the region $\frac{\tt}2 \leq \tt-\rs \leq 2\tt$ and identically one when $\frac{3\tt}4 \leq \tt-\rs \leq \tt$ . We now write $G_{\alpha}$ in the form
\[
G_{\alpha} =  g_1 \phi_{\leq |\alpha|+6} +  \partial_t (g_2 \phi_{< |\alpha|+6}), 
\qquad g_1 \in S^Z(r^{-3}), \ \ g_2 \in S^Z_{der}(r^{-2}).
\]
Here we can confine ourselves to $\partial_t$ derivatives in the last term
because for any $S$ and $\Omega$ component we gain a factor of $r^{-1}$
and include it in the first term.
We now split  $G_{\alpha}$ into two parts, so that later on we can apply Lemma~\ref{Minkdcy}:
\[
G_{\alpha}= G_{\alpha}^1+  \partial_t G_{\alpha}^2,
\]
with
\[
G_{\alpha}^1 = g_1 \phi_{\leq |\alpha|+6} +  \partial_t  \left((1-\chi_0) g_2 \phi_{< |\alpha|+6}\right),
\quad G_{\alpha}^2 =  \chi_0 g_2 \phi_{< |\alpha|+6}.
\]

Let $q_1 = \min\{q, pq - 1\}$; more precisely, $q_1=q$ when $p\geq 4$, and $q_1=3q-1$ when $p=3$. Using \eqref{rtot} and \eqref{dercone}, we see that
\beq\label{ptdecayu2}
|\phi_{\leq m+n}| + | \la r \ra \nabla \phi_{\leq m+n}|\lesssim \frac{1}{\la \tt\ra\la \tt-\rs\ra^{q_1}},
\eq
\[
|\nabla \phi_{\leq m+n}| \lesssim \frac{1}{\la \tt\ra\la \tt-\rs\ra^{1+q_1}}.
\]
In particular, we have, for $|\alpha|\leq m+n-12$
\beq\label{G12est}\begin{split}
& \| G_{\alpha}^1 \|_{L_{3,1, q_1}^{\infty}} \lesssim \|\la\tt\ra \la\tt-\rs\ra^{q_1} \phi_{\leq |\alpha|+12}\|_{L^\infty L^\infty} \lesssim \varepsilon \\ &
\|G_{\alpha}^2\|_{L_{2,1, q_1}^{\infty}} + \|SG_{\alpha}^2\|_{L_{2,1, q_1}^{\infty}} + \|\Omega G_{\alpha}^2\|_{L_{2,1,q_1}^{\infty}}+ \|\la \tt-\rs\ra \pa G_{\alpha}^2\|_{L_{2,1,q_1}^{\infty}} \lesssim \|\la\tt\ra \la\tt-\rs\ra^{q_1} \phi_{\leq |\alpha|+12}\|_{L^\infty L^\infty} \lesssim \varepsilon.
\end{split}\eq

Consider first the case $p=3$. We note that, due to \eqref{ptdecayu2}
\[
F_{\alpha}\lesssim \frac{\varepsilon}{\la \tt\ra^3 \la \tt-\rs\ra^{3q_1}} \lesssim \frac{\varepsilon}{\la r\ra^2 \la \tt\ra \la \tt-\rs\ra^{3q_1}}.
\]
Lemma~\ref{Minkdcy} (with $\beta=2$, $\gamma=1$ and $\eta = 3q_1$) yields, provided that we pick $q_1>1/3$ (which requires $q>4/9$):  
\beq\label{p3phi2}
\psi_2 \lesssim \frac{\varepsilon}{\la r\ra \la \tt-\rs\ra}.
\eq
On the other hand, due to \eqref{G12est} and \eqref{ptdecayu2}, we can apply Lemma~\ref{Minkdcy} with $\beta=3$, $\gamma=1$ and $\eta = q_1$. We obtain
\beq\label{p3phi3}
\psi_1 \lesssim \frac{\varepsilon}{\la r\ra \la \tt-\rs\ra^{1+q_1-\delta}} \lesssim \frac{\varepsilon}{\la r\ra \la \tt-\rs\ra}.
\eq

We get by \eqref{ptdecayu2}, \eqref{p3phi2} and \eqref{p3phi3} that
\[
\phi_{\alpha} \lesssim \frac{\varepsilon}{\la r\ra \la \tt-\rs\ra}.
\]

Using \eqref{rtot} we can turn the $r$ into a $t$ for $|\alpha|\leq m$:
\beq\label{p3final}
\phi_{\alpha} \lesssim \frac{\varepsilon}{\la \tt\ra \la \tt-\rs\ra}.
\eq
which finishes the proof when $p=3$. We remark that it is the nonlinearity that dictates the rate of the decay in this case.

Assume now that $p\geq 4$. In this case, \eqref{ptdecayu2} and \eqref{G12est} imply
\beq\label{FGq}\begin{split}
& F_{\alpha} \lesssim \frac{\varepsilon}{\la \tt\ra^p \la \tt-\rs\ra^{pq}} \lesssim \frac{\varepsilon}{\la r\ra^3 \la \tt\ra \la \tt-\rs\ra^q}, \qquad \| G_{\alpha}^1 \|_{L_{3, 1, q}^{\infty}} \lesssim \varepsilon \\ & \|G_{\alpha}^2\|_{L_{2,1, q}^{\infty}} + \|SG_{\alpha}^2\|_{L_{2,1, q}^{\infty}} + \|\Omega G_{\alpha}^2\|_{L_{2,1, q}^{\infty}}+ \|\la \tt-\rs\ra \pa G_{\alpha}^2|\|_{L_{2,1, q}^{\infty}} \lesssim \varepsilon.
\end{split}\eq

Lemma~\ref{Minkdcy} (with $\beta=3$ and $\eta = q$) yields
\beq\label{phi23}
\psi_2, \psi_3 \lesssim \frac{\varepsilon}{\la r\ra \la \tt-\rs\ra^{1+q}}, \qquad q<1/2.
\eq

We thus obtain by \eqref{ptdecayu2} and \eqref{phi23}
\[
\phi_{\alpha} \lesssim \frac{\varepsilon}{\la r\ra \la \tt-\rs\ra^{1+q}}, \qquad q<1/2,
\]
and by \eqref{rtot} we can turn the $r$ into a $t$:
\beq\label{ptdecayu3}
\phi_{\alpha} \lesssim \frac{\varepsilon}{\la t\ra \la \tt-\rs\ra^{1+q}}, \qquad q<1/2.
\eq

We now iterate one last time. Due to \eqref{ptdecayu2} and \eqref{G12est}, we obtain
\[\begin{split}
& F_{\alpha} \lesssim \frac{\varepsilon}{\la \tt\ra^p \la \tt-\rs\ra^{pq+p}} \lesssim \frac{\varepsilon}{\la r\ra^3 \la \tt\ra \la \tt-\rs\ra^{q+1}}, \qquad \| G_{\alpha}^1 \|_{L_{3, 1, q+1}^{\infty}} \lesssim \varepsilon \\ & \|G_{\alpha}^2\|_{L_{2,1, q+1}^{\infty}} + \|SG_{\alpha}^2\|_{L_{2,1, q+1}^{\infty}} + \|\Omega G_{\alpha}^2\|_{L_{2,1, q+1}^{\infty}}+ \|\la \tt-\rs\ra \pa G_{\alpha}^2|\|_{L_{2,1, q+1}^{\infty}} \lesssim \varepsilon.
\end{split}\]

Lemma~\ref{Minkdcy} (with $\beta=3$ and $\eta = q+1$) yields
\beq\label{phi23bis}
\phi_2, \phi_3 \lesssim \frac{\varepsilon}{\la r\ra \la \tt-\rs\ra^2}.
\eq

We thus obtain by \eqref{ptdecayu3} and \eqref{phi23bis}
\[
\phi_{\alpha} \lesssim \frac{\varepsilon}{\la r\ra \la \tt-\rs\ra^2}, 
\]
and by \eqref{rtot} we can turn the $r$ into a $t$:
\beq\label{ptdecayu4}
\phi_{\alpha} \lesssim \frac{\varepsilon}{\la \tt\ra \la \tt-\rs\ra^2},
\eq
which finishes the proof.

\end{document}